\documentclass[11pt]{article}
\usepackage{amsfonts}
\usepackage{color,xcolor}
\usepackage{amsmath}
\usepackage{latexsym,amsfonts,amssymb}
\usepackage{mathrsfs}
 \usepackage[pagewise]{lineno}

\usepackage{amsthm,cite}

\newtheorem{thm}{\bf Theorem}[section]
\newtheorem{lem}[thm]{\bf Lemma}

\newtheorem{cor}[thm]{\bf Corollary}
\usepackage[hidelinks,colorlinks=true,linkcolor=black,citecolor=black]{hyperref} 
\newtheoremstyle{REMARK}{3pt}{3pt}{\rm}{}{\bf}{.}{0.5em}{}
\theoremstyle{REMARK}\newtheorem{rem}[thm]{Remark}
\newtheorem*{rem*}{Remark} \newtheorem*{ac*}{Acknowledgments}

\usepackage{CJKutf8}

\begin{document}
\begin{CJK}{UTF8}{gkai}

\title{Analysis of a nonlinear free-boundary tumor model with three layers}
\author{
Junde Wu$^\dagger$,\;
Hao Xu$^\dagger$\footnote{Corresponding author. \newline
   \qquad E-mails: wujund@suda.edu.cn (J. Wu),    
   1158637462@qq.com (Hao Xu), zdzyh@outlook.com(Yuehong Zhuang)} \;
   and
   Yuehong Zhuang$^\ddagger$
}

\date{$^\dagger$School of Mathematical Sciences, Soochow University,
  \\ [0.1 cm]
  Suzhou,   Jiangsu 215006, PR China
  \\ [0.2 cm]
  $^\ddagger$Department of Mathematics, Jinan University,
  \\ [0.1 cm]
  Guangzhou, Guangdong 510632, PR China
  }
  \medskip

\maketitle
\begin{abstract} 
\noindent In this paper, we study a nonlinear free boundary problem modeling the growth of spherically symmetric tumors. The tumor consists of a central necrotic core, an intermediate annual quiescent-cell layer, and an outer proliferating-cell layer. The evolution of tumor layers and the movement of the tumor boundary are totally governed by external nutrient supply and conservation of mass. 
The three-layer structure generates three free boundaries with boundary conditions of different types.
We develop a nonlinear analysis method to get over the great difficulty arising from free boundaries and the discontinuity of the nutrient-consumption rate function. By carefully studying the mutual relationships between the free boundaries, we reveal the evolutionary mechanism in tumor growth and the mutual transformation of its internal structures. The existence and uniqueness of the radial stationary solution is proved, and its globally asymptotic stability towards different dormant tumor states is established.

\smallskip


\smallskip


\medskip

\end{abstract}

\newpage 

\section{Introduction}
\setcounter{equation}{0}
\hskip 1em

In this paper we study the following free boundary problem modeling the growth of
spherically symmetric tumors with three-layer structure:
\begin{equation}\label{1.1}
  \Delta_r \sigma = f(\sigma)\chi_{\{\sigma>\sigma_Q\}}+g(\sigma)
  \chi_{\{\sigma_D<\sigma\le\sigma_Q\}}\qquad\mbox{for}\;\; 0<r<R(t),\;\;t>0,
\end{equation}
\begin{equation}\label{1.2}
     \sigma_r(0,t)=0,\quad \sigma\big(R(t),t\big)=\bar\sigma \qquad \mbox{for}\;\;t>0, \qquad
\end{equation}
\begin{equation}\label{1.3}
     \displaystyle
  R'(t)R^2(t)=\int_0^{R(t)} \Big(S\big(\sigma(r,t)\big)\chi_{\{\sigma>\sigma_Q\}}-\nu_1\chi_{\{\sigma_D<\sigma\le\sigma_Q\}}-\nu_2\chi_{\{\sigma\le\sigma_D\}}\Big)r^2dr
 \quad\mbox{for}\;\;t>0,
\end{equation}
\begin{equation}\label{1.4}
      R(0)=R_0,
\end{equation}
  where $\sigma(r, t)$ and $R(t)$ are both unknown functions representing  the  concentration of  nutrients and  the  tumor radius at time $t>0$,  $\chi_E$ is the indicator function on a set $E$, namely $\chi_E(x)=1$ for $x\in E$ and 
  $\chi_E(x)=0$ for $x\notin E$. Constants $\sigma_Q$ and $\sigma_D$ are two positive nutrient concentration threshold values for distinguishing between the proliferating phase and the quiescent phase, and between the quiescent phase and the necrotic phase, respectively. 
   It makes that the region 
   $\{\sigma(r, t)>\sigma_Q\}$ is the proliferating layer with only proliferating cells,
   $\{\sigma_D<\sigma(r,t)\le\sigma_Q\}$
   is the quiescent layer with only quiescent cells,
   and $\{\sigma(r,t)\le\sigma_D\}$ is the necrotic core with only dead cells. $f(\sigma)$ and $g(\sigma)$ are two given functions representing the nutrient consumption rate functions for proliferating cells and quiescent cells, respectively,  $S(\sigma)$ is the volume growth rate function of proliferating cells. Constants
   $\bar\sigma$, $\nu_1$ and $\nu_2$  are all positive, and $\bar\sigma$ represents the external nutrient supply, $\nu_1$ and $\nu_2$ represent the removal rates for quiescent and necrotic cells, respectively.  Finally, $R_0>0$ is the initial tumor radius.   

   For simplicity of model computation and analysis, $f$, $g$ and $S$ are typically taken as constant functions or  linear functions with the form of 
\begin{equation}\label{1.5}
f(\sigma)=\lambda_1\sigma,\qquad g(\sigma)=\lambda_2\sigma,\qquad  
S(\sigma)=\mu(\sigma-\tilde\sigma),
\end{equation}
where  $\mu,\tilde\sigma,\lambda_1,\lambda_2$ are all positive constants (cf.\cite{byr-cha-96, byr-cha-97, cui-fri-01}).    
In this paper, we consider general nonlinear functions with the following assumptions:
 
\vspace{0.3em}

 \noindent $(A1)$  
 $f$, $g\in C^1[0,+\infty)$, $f'>0$, $g'>0$, $\displaystyle\sup_{[0,+\infty)} f'(x)$, $\displaystyle \sup_{[0,+\infty)} g'(x)<+\infty$ 
 and $f(0)=g(0)=0$.

 \vspace{0.1em}
  
 \noindent $(A2)$  
 $S\in C^1[0,+\infty)$, $S'>0$ and  $S(\tilde\sigma)=0$  for some $\tilde \sigma>0$.
 
  \vspace{0.1em}

\noindent $(A3)$ 
  $0< \sigma_D<\sigma_Q<\tilde\sigma$,
    $f(\sigma_Q)\ge g(\sigma_Q)$,  $S(\sigma_Q)\ge-\nu_1\ge -\nu_2$.
\vspace{0.3em}

These assumptions are all biologically meaningful. $(A1)$ and $(A2)$ mean that  nutrient consumption rate functions $f$ and $g$ and the volume growth rate function $S$ are all strictly  increasing  in the nutrient concentration. The constant $\tilde\sigma$ can be regarded as the nutrient concentration threshold at which the  birth rate and the death rate of proliferating cells are in balance, and the first inequality in $(A3)$ is natural (cf. \cite{byr-cha-97}).  The second  inequality in $(A3)$ means that proliferating cells consume nutrients faster than quiescent cells at their threshold concentration $\sigma_Q$. Since $-\nu_1$ and $-\nu_2$ can be regarded as the volume growth rate of quiescent cells and necrotic cells, respectively, 
the last inequality in $(A3)$ means that proliferating cells always grow faster than quiescent cells, and necrotic cells are removed more rapidly than quiescent cells. For more discussion on these assumptions, see \cite{cui-05, wu-wang-19, zheng-li-zhuang}. 

Problem \eqref{1.1}--\eqref{1.4} is a generalized three-layer tumor model suggested by Byrne and Chaplain \cite{byr-cha-96}. 
In the limiting case $\sigma_Q=\sigma_D=0$, it becomes the classical one-layer tumor model with only proliferating cells which has been extensively studied, many illuminating results such as asymptotic behavior of radial solutions and non-radial solutions, the 
existence of symmetry-breaking bifurcation stationary solutions and Hopf bifurcations have been
well established, we refer to \cite{cui-05,cui-09,cui-esc-07,cui-esc-08,fri-hu-06,fri-hu-08,fri-rei-99,he-xing-hu,hua-hu-24,hua-zha-hu-19, zhao2025} and references cited therein. 
In the case $\sigma_Q=\sigma_D>0$, it can be regarded as a two-layer necrotic tumor model, for the existence of radial stationary solutions, asymptotic stability of radial stationary solutions under radial or non-radial perturbations, and the
existence of non-flat
bifurcation stationary solutions, we refer readers to
see
\cite{bue-erc-08,
cui-06, cui-fri-01, wu-wang-19, wu-xu-20, xu-zhang-zhou,lu-hao-hu,wu-18,wu-19,wu-21}.
In the case $\sigma_Q>\sigma_D=0$, this problem can be regarded as another two-layer tumor model which contains a quiescent core and an outer shell of proliferating cells.  Liu and Zhuang  studied the asymptotic behavior in \cite{liu-zhuang} and time-delay effects in \cite{liu-z-2} with linear functions  (1.5). Recently, Wu, Xu and Zhuang \cite{wu-xu} established the existence and asymptotic stability of radial stationary solutions for the nonlinear consumption rate and proliferation rate functions, by thoroughly analyzing the relationships between model variables.

For the three-layer tumor model, Byrne and Chaplain \cite{byr-cha-97} first considered a simple case $f(\sigma)=g(\sigma)\equiv \lambda_0$
(a positive constant), and linear stability analysis and numerical simulation were carried out. Zheng, Li and Zhuang \cite{zheng-li-zhuang}  studied the case $f(\sigma)=g(\sigma)=\lambda\sigma$,  where the quiescent layer and proliferating layer can be handled as one layer together in mathematical analysis. Liu and Zhuang \cite{liu-z-3} further considered  $f(\sigma)=\delta_1+\lambda\sigma$ and $g(\sigma)\equiv\delta_2$ with positive constants $\delta_1,\delta_2,\lambda$ satisfying $\delta_1+\lambda\sigma_Q>\delta_2$.
By careful computation with explicit 
expressions of nutrient concentration $\sigma$ in the radius  $R$, they established the asymptotic stability of the unique radial stationary solution. 
 However, in reality, the proliferating cells and the quiescent cells have different and complex mechanisms of nutrient consumption and cell growth.  The formation of necrotic cores with distinct 
 multi-layered configuration in tumor growth is a basic and interesting problem in modeling and analysis which has been explored for several decades (cf. \cite{byr-cha-97, greenspan, low, perthame}).  Nonlinear nutrient consumption rate and cell growth rate
 functions should be considered in necrotic tumor models for an in-depth understanding of tumor growth in early stages.

  In this paper, we aim to rigorously study
  the interactions among different tumor  layers and asymptotic behavior of
  radial solutions of problem \eqref{1.1}--\eqref{1.4} under assumptions $(A1)$--$(A3)$. 
  Note that in nonlinear case the
 nutrient concentration $\sigma$ cannot be solved explicitly in $R$ any
more. 
The volume growth rate function and the consumption rate function have discontinuity  
across the inner two free boundaries.
 Comparing with two-layer tumor models,  the three-layer model features three free boundaries with boundary conditions of different types, which also gives rise to many new challenges. For instance, we need to address several different elliptic free boundary problems and a new nonlinear critical problem, see Lemma \ref{lem2.3} 
 and Lemma \ref{lem2.4}.
 The potential relations between these
 three free boundaries become very complicated and we need to provide some 
 insights into the growth mechanisms of these layers with different types of 
 tumor cells.

We shall develop an inside to outside
method to overcome these difficulties. We first solve a Cauchy problem for $\sigma$ in the region $\{r>\rho\}$ with any given necrotic radius $\rho>0$,
and using the shooting method to get the quiescent radius $\eta=\eta(\rho)$. Then by using the continuity of nutrient flux  across the boundary
$r=\eta$, we continue to solve another  Cauchy problem for
$\sigma$ in the region $\{r>\eta\}$ and similarly get the tumor radius $R=R(\eta)$. To study the relationships between $\rho$, $\eta$, $R$ and solutions of Cauchy problems
on different model parameters, we 
carefully choose boundary value conditions and apply the linearization method to related elliptic problems, based on the maximum principle. With some delicate arguments, we completely figure out various dependence relationships between three free boundaries and the external nutrient supply $\bar\sigma$.  We finally find two critical nutrient values $\sigma^*$ and $\sigma_*$ with $\sigma^{*}>\sigma_*>\tilde\sigma$ such that free boundary problem \eqref{1.1}--\eqref{1.4} has a unique three-layer  stationary solution if and only if $\bar\sigma>\sigma^{*}$, and has 
a unique two-layer proliferating-quiescent stationary solution if and only if
$\sigma_*<\bar\sigma\le \sigma^*$, and
has a unique one-layer proliferating stationary 
solution if and only if 
$\tilde\sigma<\bar\sigma\le \sigma_*$. Moreover, we establish the global asymptotic stability of all these stationary solutions. 
It is worthy of note that
our method based on the shooting method and
the linearization method to elliptic problems layer by layer from the inside to outside is also applicable for similar multi-layer problems.

The outline of the rest of this paper is as follows.
In Section \ref{2}, we give the existence and uniqueness of  stationary solutions of problem \eqref{1.1}--\eqref{1.4}. 
In Section \ref{3}, we establish the global well-posedness  of problem (\ref{1.1})--$(\ref{1.4})$
 and  the asymptotic stability  of  stationary solutions. In the last section, we draw a conclusion and give some biological implications. 

\medskip
\hskip 1em

\section{Stationary solutions}\label{2}
\setcounter{equation}{0}
\hskip 1em

In this section, we study the existence and uniqueness of stationary solutions of problem $(\ref{1.1})$--$(\ref{1.4})$. Clearly, the stationary 
solutions fall into three distinct types. The stationary solution with a
one-layer structure is denoted by $(\sigma_s,R_s)$, for the dormant tumor consists entirely of proliferating cells, which satisfies

\begin{equation}\label{2.3}
\left\{
\begin{array}{l}
  \displaystyle\sigma''(r)+\frac{2}{r}\sigma'(r) = f(\sigma)
  \qquad\mbox{for}\;\; 0<r<R,
\\[0.3 cm]
  \sigma'(0)=0, \quad \sigma(0)\ge\sigma_Q, \quad
  \sigma(R)=\bar\sigma,
\\[0.3 cm]
  \displaystyle
  \int_0^R S(\sigma(r))r^2dr=0.
\end{array}
\right.
\end{equation}

Another type is the stationary solution with a two-layer structure, which  is denoted by 
 $(\sigma_s,\eta_s,R_s)$ and represents a dormant tumor with a quiescent core whose radius is $\eta_s$, surrounded by a proliferating shell with radius $R_s$. It satisfies
\begin{equation}\label{2.2}
\left\{
\begin{array}{l}
   \displaystyle\sigma''(r)+\frac{2}{r}\sigma'(r) = g(\sigma)
  \qquad\mbox{for}\;\; 0<r<\eta,
\\[0.3 cm]
 \sigma'(0)=0, \quad \sigma(0)\ge\sigma_D,
 \quad
  \sigma(\eta)=\sigma_Q,
  \\[0.3 cm]
  \displaystyle\sigma''(r)+\frac{2}{r}\sigma'(r) = f(\sigma)
  \qquad\mbox{for}\;\; \eta<r<R,
\\[0.3 cm]
\sigma'(\eta-0)=\sigma'(\eta+0),\quad \sigma(R)=\bar\sigma, 
\\[0.3 cm]
  \displaystyle
  \int_{\eta}^R S(\sigma(r))r^2dr-\int_0^\eta \nu_1r^2dr=0.
\end{array}
\right.
\end{equation}

 The last type is the stationary solution with a three-layer structure, which is denoted by  $(\sigma_s,\rho_s,\eta_s,R_s)$ and represents a dormant tumor with a necrotic core whose radius is $\rho_s$, an intermediate quiescent layer whose radius is $\eta_s$ and an outer proliferating shell with radius $R_s$.    It satisfies the following problem: 
\begin{equation}\label{2.1}
\left\{
\begin{array}{l}
\sigma(r)=\sigma_D \qquad \mbox{for}\;\; 0\le r\le \rho, 
\\[0.3cm]
  \displaystyle\sigma''(r)+{\frac2r}\sigma'(r) = g(\sigma)
  \qquad\mbox{for}\;\; \rho<r<\eta,
\\[0.3 cm]
\sigma'(\rho)=0,\qquad\sigma(\eta)=\sigma_Q,
\\[0.3 cm]
 \displaystyle\sigma''(r)+{\frac2r}\sigma'(r) = f(\sigma)
  \qquad\mbox{for}\;\; \eta<r<R,
\\[0.3 cm]
  \sigma'(\eta-0)=  \sigma'(\eta+0),\quad
  \sigma(R)=\bar\sigma,
\\[0.3 cm]
  \displaystyle
  \int_{\eta}^R S(\sigma(r))r^2dr-\int_{\rho}^{\eta}\nu_1 r^2dr-\int_0^{\rho}\nu_2 r^2dr=0.
\end{array}
\right.
\end{equation}

The above one-layer and two-layer stationary solutions 
without the constraints on $\sigma(0)$
have been well studied; see \cite{cui-05, wu-xu}. 
Therefore, we mainly focus on the existence and uniqueness of the three-layer stationary solution 
of problem $(\ref{1.1})$--$(\ref{1.4})$.

We first investigate the following initial value problem:
\begin{equation}\label{2.4}
\left\{
\begin{array}{l}
  \displaystyle u''(r)+\frac{2}{r}u'(r) = g(u(r))
  \qquad\mbox{for}\;\; r>\rho,
\\[0.3 cm]
u(\rho)=\sigma_D,\qquad u'(\rho)=0.
\end{array}
\right.
\end{equation}

\medskip

\begin{lem}\label{lem2.1}
       Under assumption $(A1)$,
       for any given  $\rho\ge0$, problem $(\ref{2.4})$
       admits a unique  solution $u=U_1(r,\rho)\in C^2[\rho,+\infty)$ 
       with the following properties:
    
    $(i)$ $U_1$ is strictly increasing and strictly convex  in  $r$, and satisfies
    $$\displaystyle \lim_{r\to +\infty} U_1(r,\rho)=+\infty.$$
    
    $(ii)$ $U_1$ and $(U_1)_r$ are  both strictly decreasing in  $\rho$, i.e., 
$$\frac{\partial U_1}{\partial \rho}(r,\rho)<0,\qquad
  \frac{\partial^2 U_1}{\partial\rho\partial r}(r,\rho)<0\qquad\mbox{for}\;\; r>\rho,\; \rho>0.
$$

\end{lem}

\begin{proof}
  The local existence and uniqueness of solutions to problem \eqref{2.4} can be proved by using a Banach fixed point argument, similarly to the proof of Lemma 2.2 in \cite{wu-xu}.
  The global existence is guaranteed by the global
  Lipschitz continuity of $g$ due to $(A1)$. Hence problem
  \eqref{2.4} has a unique global solution $u = U_1(r, \rho)$ for $r\in[\rho,+\infty)$. Clearly, $U_1(r, \rho)\ge\sigma_D$.
 Then we have  $g(U_1(r,\rho))\ge g(\sigma_D)$, which together with integrating $(\ref{2.4})$ implies
\begin{equation}\label{daoshu}
 u'(r)=\frac{1}{r^2}
 \int_{\rho}^r g(U_1(\tau,\rho))\tau^2d\tau\ge g(\sigma_D)\frac{r^3-\rho^3}{3r^2}>0
 \qquad\mbox{for}\;\;r> \rho.
\end{equation}
Combining $(A1)$ with $(\ref{daoshu})$, we derive that 
\begin{equation*}
    \lim\limits_{r\to+\infty}U_1(r,\rho)=\lim\limits_{r\to+\infty}(U_1)_r(r,\rho)=+\infty.
\end{equation*}
By  $(A1)$, $(\ref{2.4})_1$ and $(\ref{daoshu})$, we further get
\begin{equation}
\begin{matrix}
\begin{aligned}
 u''(r) -\displaystyle\frac{u'(r)}{r} &=g(U_1(r,\rho))
  -\displaystyle\frac{3}{r^3}\int_{\rho}^r g(U_1(\tau,\rho))d\tau  \\     
  &> g(U_1(r,\rho))-\frac{3}{r^3} g(U_1(r,\rho))\frac{1}{3}(r^3-\rho^3)\\
  &=g(U_1(r,\rho))\frac{\rho^3}{r^3}\ge 0.
\end{aligned}   
\end{matrix}
\end{equation}
Thus assertion $(i)$ follows.
Finally, we observe that for $r>\rho$,
\begin{equation}\label{Ur}
    r^2 (U_1)_r(r,\rho)=\int_\rho^r  g( U_1(\tau,\rho))\tau^2 d\tau.
\end{equation}
By $(A1)$, $(\ref{Ur})$  and a comparison argument with
 slight modifications of the proof of 
Lemma 2.2 $(iii)$ in \cite{wu-xu},
 we get assertion $(ii)$.  
 \end{proof}

Given $\rho\ge 0$ and
  $\sigma_Q>\sigma_D$,  Lemma \ref{lem2.1} ensures the existence and uniqueness of $\eta=\eta(\rho)$ $\in(\rho,+\infty)$ such that
\begin{equation}\label{shoot}
U_1(\eta(\rho),\rho)=\sigma_Q.
\end{equation}
Moreover,  $\eta(\rho)$ is strictly increasing on $(0,+\infty)$,
which implies that there exists  $\eta^*>0$ such that
\begin{equation}\label{eta*} 
\lim_{\rho\to 0^+} \eta(\rho)=\eta^*.
\end{equation}
Clearly,  $\eta^*>0$, and 
 $\eta^*$ is the critical radius of the quiescent core in the two-layer problem \eqref{2.2} where $\sigma(0)=\sigma_D$. 

By the strict monotonicity of $\eta=\eta(\rho)$, we infer that the mapping $\rho\mapsto \eta(\rho)$ is a 1-1 correspondence from $[0,+\infty)$ to $[\eta^*,+\infty)$. For convenience, we rewrite $\rho=\rho(\eta)$ for  $\eta\in[\eta^*,+\infty)$, and denote  
$$
u(r)=U_1(r,\rho(\eta))=:\widetilde U_1(r,\eta)
\qquad\mbox{for}\;\;\rho(\eta)<r<\eta,\,\eta\geq\eta^*.
$$ 
In summary, given any $\eta>\eta^*$, the function $\widetilde U_1(r,\eta)$ ($\rho(\eta)<r<\eta$) uniquely solves $(\ref{2.1})_2$--$(\ref{2.1})_3$.

Based on $(\ref{daoshu})$, we define
\begin{equation}\label{2.10} 
\Phi(\eta):=(\widetilde U_1)_r(\eta,\eta)=\frac{1}{\eta^2}\int_{\rho(\eta)}^{\eta}g(\widetilde U_1(\tau,\eta))\tau^2d\tau
\quad\mbox{for}\;\; \eta\ge\eta^*.
\end{equation}
By $(A1)$ and Lemma 2.1 $(i)$,  
\begin{equation}\label{2.11}
0<\Phi(\eta)<\frac{1}{3}g(\sigma_Q)\eta \quad \mbox{for} \;\;\eta\ge\eta^*.
\end{equation}
Moreover, we claim that 
\begin{equation}\label{2.12}
\Phi'(\eta)>0 \qquad \mbox{for}\;\; \eta>\eta^*.
\end{equation}
In fact, by $(\ref{2.1})_2$--$(\ref{2.1})_3$, we  see that for every $\eta>\eta^*$ the function $u_{\eta}(r):=\frac{\partial \widetilde U_1}{\partial \eta}(r,\eta)$ satisfies the following elliptic problem
\begin{equation}\label{2.13}
\left\{
\begin{array}{l}
  \displaystyle(u_{\eta})''+\frac{2}{r}(u_{\eta})' = g'(\widetilde U_1)u_{\eta}
  \qquad\mbox{for}\;\; \rho<r<\eta,
\\[0.3 cm]
u_{\eta}(\rho)=0,\qquad u_{\eta}(\eta)=-(\widetilde U_1)_r(\eta,\eta)<0,
\end{array}
\right.
\end{equation}
where $\rho=\rho(\eta)$.
  Then by the strong  maximum principle, 
\begin{equation}\label{2.14}
  u_\eta(r)=\frac{\partial \widetilde U_1}{\partial \eta}(r,\eta)<0 \qquad \mbox{for} \;\; \rho<r<\eta.
    \end{equation}
Moreover, the function $u_r(r):=\frac{\partial \widetilde U_1}{\partial r}(r,\eta)$ satisfies the following system
\begin{equation*}
\left\{
\begin{array}{l}
  \displaystyle(u_r)''+\frac{2}{r}(u_r)'
  = g'(\widetilde U_1) u_r+\frac{2}{r^2}u_r
  \qquad\mbox{for}\;\; \rho<r<\eta,
\\[0.3 cm]
u_r(\rho)=0,\qquad u_r(\eta)=(\widetilde U_1)_r(\eta,\eta)>0,
\end{array}
\right.
\end{equation*}
where $\rho=\rho(\eta)$.
Denote $w(r)=u_r(r)+u_{\eta}(r)$.
  It satisfies
\begin{equation*}
\left\{
\begin{array}{l}
  \displaystyle w''(r)+\frac{2}{r}w'(r)= g'(\widetilde U_1)w(r)+\frac{2}{r^2} u_r
  \qquad\mbox{for}\;\; \rho<r<\eta,
\\[0.3 cm]
w(\rho)=0,
\qquad
w(\eta)=0.
\end{array}
\right.
\end{equation*}
Note that $g'(\widetilde U_1)>0$ and $u_r(r)>0$ for $r>\rho$. Then by applying the strong maximum principle and Hopf lemma, we obtain
\begin{equation}\label{2.15}
w'(\eta) 
=(\widetilde U_1)_{rr}(\eta,\eta)+(\widetilde U_1)_{r\eta}(\eta,\eta)=\Phi'(\eta)>0.
\end{equation}
This proves \eqref{2.12}.

Next, we proceed to solve problem $\eqref{2.1}_4$--$\eqref{2.1}_5$ by considering the following initial value problem: 
\begin{equation}\label{2.16}
\left\{
\begin{array}{l}
 \displaystyle u''(r)+\frac{2}{r}u'(r) = f(u(r))
  \qquad\mbox{for}\;\; r>\eta,
\\[0.3 cm]
  u(\eta)=\sigma_Q,
\\[0.3 cm]
  u'(\eta)=  \Phi(\eta).
\end{array}
\right.
\end{equation}

\medskip

\begin{lem}\label{lem2.2}
Under assumptions $(A1)$ and $(A3)$,  for any given $\eta\ge\eta^*$,  problem $(\ref{2.16})$
   has a unique solution $u= U_2(r,\eta)\in
   C^2\left[\eta, +\infty\right)$ with the following properties:

 $(i)$ $ U_2$ is strictly increasing and strictly convex in $r$, i.e.,
$$
  (U_2)_{rr}(r,\eta) >\frac{1}{r}(U_2)_r(r,\eta)> 0 \qquad\mbox{for}\;\;
  r>\eta,
$$
  and satisfies
  $$\displaystyle \lim_{r\to +\infty}U_2(r,\eta)=+\infty.$$

$(ii)$ $U_2$ is strictly decreasing
in $\eta$,  i.e., 
$$
\frac{\partial U_2}{\partial \eta}(r,\eta)<0 \qquad \mbox{for} \;\; r>\eta,\; \eta>\eta^*.
$$
\end{lem}

\begin{proof}

The existence and uniqueness of solutions to problem $(\ref{2.16})$ and assertion $(i)$ can be easily verified similarly 
as Lemma \ref{lem2.1} $(i)$, so we only need to prove assertion $(ii)$.
For any given $\eta>\eta^*$, define 
$$
z(r):=\displaystyle\frac{\partial U_2}{\partial \eta}(r,\eta).
$$
Then by $(\ref{2.10})$, $(\ref{2.11})$ and $(\ref{2.16})$, we have
\begin{equation}\label{2.17}
 \left\{
    \begin{array}{l}
     z''(r)+\displaystyle\frac{2}{r}z'(r)=f'(U_2)z(r)\qquad \mbox{for} \;\; r>\eta,   
     \\[0.3cm]
       z(\eta)=-\Phi(\eta)<0,
    \\[0.3cm] z'(\eta)=\Phi'(\eta)+\displaystyle\frac{2}{\eta}\Phi(\eta)-f(\sigma_Q),
    \end{array}
    \right.
\end{equation}
and
\begin{equation}\label{2.18}
\Phi'(\eta)=-\frac{2}{\eta}\Phi(\eta)+g(\sigma_Q)+\Psi(\eta),
\end{equation}
where 
$$
\Psi(\eta)=-\frac{\rho^2(\eta)}{\eta^2}g(\sigma_D)\rho'(\eta)+\frac{1}{\eta^2}\int_{\rho(\eta)}^{\eta}g'\big(\widetilde U_1(\tau,\eta)\big)\frac{\partial \widetilde U_1}{\partial \eta}(\tau,\eta)\tau ^2d\tau.
$$
From $(A1 )$, $(\ref{2.14})$ and $\rho'(\eta)\ge 0$, we  have 
\begin{equation}\label{2.19}
    \Psi(\eta)<0 \qquad \mbox{for} \;\; \eta>\eta^*.
\end{equation}
Substituting $(\ref{2.18})$ into $(\ref{2.17})_3$ and using $(A3)$, $(\ref{2.19})$, we obtain
$$
z'(\eta)=g(\sigma_Q)-f(\sigma_Q)+\Psi(\eta)<0.
$$
Thus by $f'>0$ and (\ref{2.17}), we easily get 
$$
  z'(r)<0\quad\mbox{and}\quad
  z(r)
<0 \qquad \mbox{for}\;\; r>\eta.
$$
The proof is complete.
\end{proof}

Given $\bar\sigma>\sigma_Q$ and $\eta\ge\eta^*$,
Lemma \ref{lem2.2} implies that there exists 
a unique $R=R(\eta,\bar\sigma)\in(\eta,+\infty)$ such that 
\begin{equation}\label{shoot2}
U_2(R(\eta,\bar\sigma),\eta)=\bar\sigma.
\end{equation}
Then by denoting 
\begin{equation}
\widetilde\Sigma(r,\eta,\bar\sigma)=\left\{
\begin{array}{l}
\sigma_D, \qquad\qquad \mbox{for}\;\; 0\le r\le \rho(\eta),
\\[0.3cm]
  \widetilde U_1(r,\eta), \qquad \mbox{for}\;\; \rho(\eta)<r\le \eta,
  \\[0.3 cm]
  U_2(r,\eta),\qquad \mbox{for}\;\; \eta<r\le R(\eta,\bar\sigma),
\end{array}
\right.
\end{equation}
we see that for $\eta>\eta^*$ the triple 
$(\sigma, \rho, R)=(\widetilde \Sigma(r,\eta,\bar\sigma),  \rho(\eta), R(\eta,\bar\sigma))$ 
uniquely solves the problem $(\ref{2.1})_1$--$(\ref{2.1})_5$ on the interval $[0,R(\eta,\bar\sigma)]$ with $R(\eta,\bar\sigma)> \eta>\rho(\eta)>0$.

From Lemma \ref{lem2.2},  we see that  $R(\eta,\bar\sigma)$ is strictly increasing in $\eta$ for any given 
$\bar\sigma>\sigma_Q$.
Define
\begin{equation}\label{R*}
R^*(\bar\sigma):= R(\eta,\bar\sigma)\Big|_{\eta=\eta^*}
\quad \mbox{for\;\;} \bar\sigma>\sigma_Q.
\end{equation}
 Clearly, it is the critical radius such that problem $(\ref{2.2})_1$--$(\ref{2.2})_4$
 has a unique solution satisfying $\sigma(0)=\sigma_D$ for $R=R^*(\bar\sigma)$ and $\eta=\eta^*$.  Evidently,
 $R^*(\bar\sigma)>\eta^*>0$. 
 Similarly, for any fixed $\bar\sigma>\sigma_Q$,
  the mapping $\eta\mapsto R(\eta,\bar\sigma)$  is a 1-1 correspondence from $[\eta^*,+\infty)$ to $[R^*(\bar\sigma),+\infty)$.
So we can also regard $\rho$ and $\eta$ as  functions of $R$ and $\bar\sigma$, i.e.,    
 $\rho=\rho(R,\bar\sigma)$, $\eta=\eta(R,\bar\sigma)$ for
$\bar\sigma>\sigma_Q$ and $R\ge R^*(\bar\sigma)$.

Rewrite the solution
 $\sigma=\Sigma(r,R,\bar\sigma)$ for $0\le r\le R$, where
 \begin{equation*}
\Sigma(r,R,\bar\sigma)=\left\{
\begin{array}{l}
\sigma_D, \qquad\qquad\quad \,\mbox{for}\;\; 0\le r\le \rho(R,\bar\sigma),
\\[0.3cm]
  V_1(r,R,\bar\sigma), \qquad \mbox{for}\;\; \rho(R,\bar\sigma)< r\le \eta(R,\bar\sigma),
  \\[0.3 cm]
  V_2(r,R,\bar\sigma),\qquad \mbox{for}\;\; \eta(R,\bar\sigma)<r\le R,
\end{array}
\right.
\end{equation*}
with
\begin{equation*}
    V_1(r,R,\bar\sigma)=\widetilde{U}_1(r,\eta(R,\bar\sigma)),\qquad V_2(r,R,\bar\sigma)=U_2(r,\eta(R,\bar\sigma)).
\end{equation*}
  According to Lemma \ref{lem2.1} and Lemma \ref{lem2.2},  we conclude that for any $\bar\sigma>\sigma_Q$,
  problem $(\ref{2.1})_1$--$(\ref{2.1})_5$
  has a unique solution
  $(\sigma,\rho,\eta)=(\Sigma(r, R,\bar\sigma), \rho(R,\bar\sigma), \eta(R,\bar\sigma))$ if and only if  $R> R^*(\bar\sigma)$.

For any $\bar\sigma>\sigma_Q$ and $R\ge R^*(\bar\sigma)$,  define
\begin{equation}\label{F(R)}
  F(R,\bar\sigma):=\frac{1}{R^3}\Big[\int_{\eta(R,\bar\sigma)}^{R}S(V_2(r, R,\bar\sigma))r^2 dr-
  \frac{\nu_1}{3}\eta^3(R,\bar\sigma)-\frac{\nu_2-\nu_1}{3}\rho^3(R,\bar\sigma)\Big].
\end{equation}
Then problem $(\ref{2.1})$ is equivalent to equation $F(R,\bar\sigma)=0$.

Now we study the monotonicity of $F(R,\bar\sigma)$ in $R$ for fixed $\bar\sigma>\sigma_Q$. By taking variable transformation $s=r/R$, we rewrite
$$
\psi(R,\bar\sigma)=\rho(R,\bar\sigma)/R,\qquad \phi(R,\bar\sigma)=\eta(R,\bar\sigma)/R, \qquad 
\mathcal{V}(s,R,\bar\sigma)=\Sigma(sR,R,\bar\sigma).
$$
 Consider the following problem
\begin{equation}\label{2.24}
\left\{
\begin{array}{l}
 \displaystyle v''(s)+\frac{2}{s}v'(s) = R^2g(v)
  \qquad\mbox{for}\;\; \psi<s<\phi,
\\[0.3cm]
v(\psi)=\sigma_D, \quad
 v'(\psi)=0,\quad
v(\phi)=\sigma_Q,
\\[0.3 cm]
 \displaystyle v''(s)+\frac{2}{s}v'(s) = R^2f(v)
  \qquad\mbox{for}\;\; \phi<s<1,
\\[0.3 cm]
v'(\phi+0)=v'(\phi-0)= R\Phi(\phi R),
\\[0.3 cm]
  v(1)=\bar\sigma.
\end{array}
\right.
\end{equation}
  We have 

\medskip

\begin{lem}\label{lem2.3}
  Under assumptions $(A1)$, $(A3)$ and   $\bar\sigma>\sigma_Q$, problem $(\ref{2.24})$
   possesses a unique solution $(v,\psi,\phi)=(\mathcal{V}(s,R,\bar\sigma), \psi(R,\bar\sigma),\phi(R,\bar\sigma))$ for any $R\ge R^*(\bar\sigma)$. Furthermore, the solution satisfies:
   
   $(i)$ $\mathcal{V}(s,R,\bar\sigma)$ is strictly increasing and strictly convex in $s$.

  $(ii)$ $\mathcal{V}(s, R,\bar\sigma)$ is  strictly decreasing in R.

  $(iii)$ $\phi(R,\bar\sigma)$ and $\psi(R,\bar\sigma)$ are both continuous and strictly increasing for $R\ge R^*(\bar\sigma)$,
  and
  \begin{equation}\label{at R*}
  \psi(R^*(\bar\sigma),\bar\sigma)=0,\qquad \phi(R^*(\bar\sigma),\bar\sigma)=\eta^*/R^*(\bar\sigma),
  \end{equation}
  \begin{equation}\label{infty}
\lim_{R\to+\infty}\psi(R,\bar\sigma)=\lim\limits_{R\to+\infty}\phi(R,\bar\sigma)=1. 
    \end{equation}
\end{lem}


\begin{proof} By taking variable transformation in problem $(2.3)_1$--$(2.3)_5$, it is easy to verify that $(v,\psi,\phi)=(\mathcal{V}(s,R,\bar\sigma), \psi(R,\bar\sigma),\phi(R,\bar\sigma))$ is the unique
solution of problem (\ref{2.24}).
The proof of the monotonicity and the convexity of 
$\mathcal{V}(s, R,\bar\sigma)$ in $s$ is 
similar as that of Lemma \ref{lem2.1} $(i)$,  we omit it here. So we mainly show the monotonicity 
of $\mathcal{V}(s, R,\bar\sigma)$, $\psi(R,\bar\sigma)$ and $\phi(R,\bar\sigma)$ in  $R$.

Denote 
$$
  z(s)=\frac{\partial \mathcal{V}}{\partial R}(s,R,\bar\sigma),\qquad \xi=\frac{\partial \psi(R,\bar\sigma)}{\partial R}, \qquad \zeta=\frac{\partial \phi(R,\bar\sigma)}{\partial R}.
$$
By the linearization of $(\ref{2.24})$, we see 
  $z(s)$, $\xi$ and $\zeta$ satisfy the following problem:

\begin{equation}\label{2.27}
\left\{
\begin{array}{l}
  \displaystyle z''(s)+\frac{2}{s}z'(s)=R^2g'(\mathcal{V})z + 2 R g(\mathcal{V})
  \qquad\mbox{for}\;\; \psi<s<\phi,
\\[0.3 cm]  
  z(\psi)=0,\quad z'(\psi)=-R^2g(\sigma_D)\xi,\quad 
  \\[0.3cm]
 z'(\phi-0)=\Phi(\phi R)+\phi R\Phi'(\phi R)+\Psi(\phi R)R^2\zeta,
  \\[0.3cm]
   z(\phi)=-R\Phi(\phi R)\zeta,
  \\[0.3cm]
  \displaystyle z''(s)+\frac{2}{s}z'(s)=R^2f'(\mathcal{V})z + 2 R f(\mathcal{V})
  \qquad\mbox{for}\;\; \phi<s<1,
\\[0.3 cm]  
 z'(\phi+0)=\Phi(\phi R)+\phi R\Phi'(\phi R)+\big(g(\sigma_Q)-f(\sigma_Q)
  +\Psi(\phi R)\big)R^2\zeta, 
\\[0.3 cm]
  z(1)=0,
\end{array}
\right.
\end{equation}
   where $\phi=\phi(R,\bar\sigma)$ and $\psi=\psi(R,\bar\sigma)$. In fact,
by $(\ref{2.24})_1$, $(\ref{2.24})_3$ and $(\ref{2.24})_4$, we have
\begin{equation*}
\left\{
\begin{array}{l}
     v''(\phi-0)=R^2g(\sigma_Q)-\displaystyle\frac{2}{\phi}R\Phi(\phi R),
     \\[0.3cm]
v''(\phi+0)=R^2f(\sigma_Q)-\displaystyle\frac{2}{\phi}R\Phi(\phi R),
    \\[0.3cm]
    z'(\phi\pm0)=\Phi(\phi R)+R\Phi'(\phi R)(\zeta R+\phi)-v''(\phi\pm0)\zeta.
\end{array}
\right.
\end{equation*}
Combining the above relations with $(\ref{2.18})$, one can derive $(\ref{2.27})_3$ and $(\ref{2.27})_6$,
other equations of $(2.27)$ are obvious.

To prove assertions $(ii)$ and $(iii)$, we use a  contradiction argument to show that
\begin{equation}\label{dandiao3}
    z(s)=\frac{\partial \mathcal{V}}{\partial R}<0 \quad \mbox{for}\;\; \psi<s<1, \quad \xi=\frac{\partial \psi}{\partial R}>0,\quad \zeta=\frac{\partial \phi}{\partial R}>0.
\end{equation}
If $\zeta\le0$, by $(A1)$, $(\ref{2.11})$, $(\ref{2.12})$, $(\ref{2.19})$, $(\ref{2.27})_4$ and $(\ref{2.27})_6$, we see that 
$z(\phi)\ge0$ and $z'(\phi+0)>0$. On the other hand, since $f(\mathcal{V})>0$ and $f'(\mathcal{V})>0$,  by applying strong maximum principle and Hopf Lemma to $(\ref{2.27})_4$--$(\ref{2.27})_7$, we get that $z'(\phi+0)<0$. This is a contradiction and thus $\zeta>0$, consequently we have $z(\phi)<0$. Combining $(\ref{2.27})_4$, $(\ref{2.27})_5$, $(\ref{2.27})_7$ and strong maximum principle, we have $z(s)<0$ for $\phi<s<1$.
By $(\ref{2.27})_1$, $(\ref{2.27})_2$, $(\ref{2.27})_4$, $g(\mathcal{V})>0$ and $g'(\mathcal{V})>0$, we can apply strong maximum principle again to deduce that $z(s)<0$ for $\psi<s<\phi$. By Hopf lemma, we  get that $z'(\psi)<0$,  which implies that $\xi>0$.

Finally, from integrating $(\ref{2.24})_3$--$(\ref{2.24})_5$ we have
$$
  \frac{\bar{\sigma}-\sigma_Q}{R^2}=\frac{\phi^2}{R}\Big(\frac{1}  {\phi}-1\Big)\Phi(\phi R)
  +\int_{\phi}^1\frac{1}{\alpha^2}\int_{\phi}^{\alpha}s^2f(\mathcal{V}(s,R,\bar\sigma))dsd\alpha,
$$
where $\phi=\phi(R,\bar\sigma)$.
If $\lim\limits_{R\to+\infty}\phi(R,\bar\sigma)\in(0,1)$, a contradiction can be obtained by taking limit $R\to+\infty$ in the above relation with noting that  $\mathcal{V}\ge\sigma_Q>0$ for $\phi(R,\bar\sigma)<s<1$ and $(\ref{2.11})$ hold.
Thus $\displaystyle\lim_{R\to+\infty}\phi(R,\bar\sigma)=1$. 

By integrating $(\ref{2.24})_1$--$(\ref{2.24})_2$, we also have
$$
\frac{\sigma_Q-\sigma_D}{R^2}=\int^{\phi(R,\bar\sigma)}_{\psi(R,\bar\sigma)}\frac{1}{\alpha^2}\int_{\psi(R,\bar\sigma)}^{\alpha}g(\mathcal{V}(s,R,\bar\sigma))s^2dsd\alpha.
$$
Likewise, there holds
$\displaystyle\lim_{R\to+\infty}\psi(R,\bar\sigma)=1$.
\end{proof} 
 
  \bigskip
With the help of Lemma \ref{lem2.3}, we  now study the  monotonicity of $F(R,\bar\sigma)$ with respect to $R$.
By  the variable transformation $r=sR$, we rewrite
$$
  F(R,\bar\sigma)
  =  \int_{\phi(R,\bar\sigma)}^1 S(\mathcal{V}(s,R,\bar\sigma))s^2ds-\frac{\nu_1}{3}\phi^3(R,\bar\sigma)-\frac{\nu_2-\nu_1}{3} \psi^3(R,\bar\sigma)
  \qquad \mbox{for}\;\; R\ge R^*(\bar\sigma).
$$
By $(A2)$, $(A3)$ and (\ref{dandiao3}), we see that  for every $\bar\sigma>\sigma_Q$ and  $R>R^*(\bar\sigma)$,
 \begin{equation}\label{F'}
   \frac{\partial F(R,\bar\sigma)}{\partial R}=\int_{\phi(R,\bar\sigma)}^1 S'(\mathcal{V})\frac{\partial \mathcal{V}}{\partial R}s^2ds
   -(S(\sigma_Q)+\nu_1)\phi^2\frac{\partial \phi}{\partial R}-(\nu_2-\nu_1)\psi^2\frac{\partial \psi}{\partial R}
  <0.
 \end{equation}
 From Lemma \ref{lem2.3} $(iii)$, we also have 
\begin{equation}\label{2.30}
\lim_{R\to+\infty} F(R,\bar\sigma)=-\frac{\nu_2}{3}<0.
\end{equation}

 Next, we need to determine the sign of $F(R^*(\bar\sigma),\bar\sigma)$, by treating $\bar\sigma$ as the variable.
 Recall $R^*=R^*(\bar\sigma)$ in $(\ref{R*})$ and denote $\phi^*=\phi^*(\bar\sigma)=\eta^*/R^*(\bar\sigma)$ where $\eta^*$ does not depend on $\bar\sigma$. They  are both regarded as  functions of $\bar\sigma$ for $\bar\sigma\in(\sigma_Q,+\infty)$. 
 We consider the following critical problem, which characterizes the nutrient concentration at the center of a tumor containing a quiescent core, 
 is exactly $\sigma_D$:
\begin{equation}\label{2.31}
\left\{
\begin{array}{l}   \displaystyle\sigma''(r)+\frac{2}{r}\sigma'(r) = g(\sigma)
  \qquad\mbox{for}\;\; 0<r<\eta^*,
\\[0.3 cm]
 \sigma'(0)=0, \quad \sigma(0)=\sigma_D,
 \quad
  \sigma(\eta^*)=\sigma_Q,
  \\[0.3 cm]
  \displaystyle\sigma''(r)+\frac{2}{r}\sigma'(r) = f(\sigma)
  \qquad\mbox{for}\;\; \eta^*<r<R^*,
\\[0.3 cm]
\sigma'(\eta^*-0)=\sigma'(\eta^*+0),\quad \sigma(R^*)=\bar\sigma. 
\end{array}
\right.
\end{equation}
Denote
$
 \mathcal{W}(s,\bar\sigma):=\mathcal{V}(s,R^*(\bar\sigma),\bar\sigma)
  $
  for $0<s<1$.
Then $(\mathcal{W}, \phi^*,R^*)$ satisfies
\begin{equation}\label{linjie}
\left\{
\begin{array}{l}
   \displaystyle W''(s)+\frac{2}{s}W'(s) = (R^*)^2g(W)
  \qquad\mbox{for}\;\; 0<s<\phi^*,
\\[0.3 cm]
 W'(0)=0, \quad W(0)=\sigma_D,
 \quad
  W(\phi^*)=\sigma_Q,
  \\[0.3 cm]
  \displaystyle W''(s)+\frac{2}{s}W'(s) = (R^*)^2f(W)
  \qquad\mbox{for}\;\; \phi^*<s<1,
\\[0.3 cm]
W'(\phi^*-0)=W'(\phi^*+0)=R^*\Phi(\phi^*R^*),
\\[0.3cm]
W(1)=\bar\sigma.
\end{array}
\right.
\end{equation}

\begin{lem}\label{lem2.4}
  Under assumptions $(A1)$--$(A3)$ and  $\bar\sigma>\sigma_Q$, 
   problem $(\ref{linjie})$ possesses a unique solution $(\mathcal{W},\phi^*,R^*)$  satisfying the following properties:
  
  $(i)$ $\mathcal{W}(s,\bar\sigma)$ and $R^*(\bar\sigma)$ are both strictly increasing in $\bar\sigma$, and $\phi^*(\bar\sigma)$ is strictly decreasing in $\bar\sigma$, i.e.,  
 \begin{equation}\label{2.33}  \frac{\partial \mathcal{W}}{\partial \bar\sigma}(s,\bar\sigma)>0\quad \mbox{for}\;\; s\in(0,1), \qquad \frac{d\phi^*}{d\bar\sigma}(\bar\sigma)<0,\qquad \frac{dR^*}{d\bar\sigma}(\bar\sigma)>0. \end{equation}
 Moreover,
 \begin{equation*}
     \lim\limits_{\bar\sigma\to+\infty}R^*(\bar\sigma)=+\infty, \qquad \lim\limits_{\bar\sigma\to\infty}\phi^*(\bar\sigma)=0.
 \end{equation*}

 $(ii)$ Define  
  $$
  \displaystyle
  \mathcal G(\bar\sigma):=F(R^*(\bar\sigma),\bar\sigma)
  =\int_{\phi^*(\bar\sigma)}^1 S(\mathcal{W}(s,\bar\sigma))s^2ds-
  \frac{\nu_1}{3}(\phi^*(\bar\sigma))^3\qquad
  \mbox{for}\;\;\bar\sigma>\sigma_Q.
$$
Then $\mathcal{G}(\bar\sigma)$ is strictly increasing in $\bar\sigma$.

\end{lem}
\begin{proof}
Note that 
$$\displaystyle \Phi(\eta^*)=\frac{1}{(\eta^*)^2}\int_0^{\eta^*} g(\widetilde{U}_1(\tau,\eta^*))\tau^2d\tau \qquad \mbox{for}\;\; 
 \eta^*>0.
 $$
From $(\ref{shoot})$ and $(\ref{eta*})$, we obtain  that
$\eta^*=\phi^*(\bar\sigma)R^*(\bar\sigma)$
is independent of $\bar\sigma$ and  $\eta^*$ is 
actually determined by the parameters  $\sigma_Q$ and $\sigma_D$.
 Consequently,  $\Phi(\eta^*)$ is also independent of $\bar\sigma$, resulting in 
\begin{equation}\label{independence}
\frac{\partial \Phi}{\partial \bar\sigma}(\eta^*)=0. 
 \end{equation}
  From $(\ref{shoot2})$, we have 
\begin{equation}\label{2.35}
    U_2(R^*(\bar\sigma),\eta^*)=\bar\sigma.
\end{equation}
By differentiating $(\ref{2.35})$ with respect to $\bar\sigma$ and Lemma \ref{lem2.2} $(i)$, we have 
\begin{equation}\label{R^*increasing}
\frac{dR^*}{d\bar\sigma} (\bar\sigma)=\frac{1}{(U_2)_r(R^*(\bar\sigma),\eta^*)}>0.
\end{equation}
Using Lemma \ref{lem2.2} $(i)$,  the property that $R^* (\bar\sigma)$ is strictly increasing in $\bar\sigma$ and $(\ref{2.35})$,
we  easily obtain 
\begin{equation}\label{2.37}
\displaystyle\lim\limits_{\bar\sigma\to+\infty}R^*(\bar\sigma)=+\infty  .
\end{equation}
For every $\bar\sigma>\sigma_Q$,  denote 
$$\displaystyle z(s):=\frac{\partial \mathcal{W}}{\partial \bar\sigma}(s,\bar\sigma),\quad  
  \bar\zeta:=\frac{dR^*}{d\bar\sigma} (\bar\sigma),\quad \bar\xi:=\frac{d\phi^*}{d\bar\sigma}(\bar\sigma).
  $$
By differentiating $R^*(\bar\sigma)\phi^*(\bar\sigma)=\eta^*$ with respect to $\bar\sigma$,  using (\ref{R^*increasing}) and the fact that $\eta^*$ is independent of $\bar\sigma$, we have 
\begin{equation}\label{xi zeta}
    \bar\xi=-\frac{\eta^*}{(R^*(\bar\sigma))^2}\bar\zeta<0.
\end{equation}
 Thus $\phi^*(\bar\sigma)$ is strictly decreasing in $\bar\sigma$ and due to $(\ref{2.37})$,  
\begin{equation}
\lim\limits_{\bar\sigma\to+\infty}\phi^*(\bar\sigma)=0.    
\end{equation}
  With $(\ref{linjie})$ and $(\ref{independence})$, one can verify that there holds
\begin{equation}\label{fanzheng}
\left\{ 
\begin{array}{l}
 \displaystyle z''(s)+\frac{2}{s}z'(s)=(R^*)^2g'(\mathcal{W})z+2 R^* \bar\zeta g(\mathcal{W})
  \qquad\mbox{for}\;\; 0<s<\phi^*,
  \\[0.3cm]
  z'(0)=0,\qquad z(0)=0, \qquad z(\phi^*)=-R^*\Phi(\phi^*R^*)\bar\xi,
\\[0.3cm]
z'(\phi^*-0)=\bar\zeta\Phi(\phi^*R^*)+\bar\xi\Big(\displaystyle\frac{2}{\phi^*}R^*\Phi(\phi^*R^*)-(R^*)^2g(\sigma_Q)\Big),
\\[0.3cm]
  \displaystyle z''(s)+\frac{2}{s}z'(s)=
  (R^*)^2f'(\mathcal{W})z+2 R^*\bar\zeta f(\mathcal{W})
  \qquad\mbox{for}\;\; \phi^*<s<1,
\\[0.3 cm]  
z'(\phi^*+0)=\bar\zeta\Phi(\phi^*R^*)+\bar\xi\Big(\displaystyle\frac{2}{\phi^*}R^*\Phi(\phi^*R^*)-(R^*)^2f(\sigma_Q)\Big),
\\[0.3cm]
  z(1)=1.\qquad 
\end{array}
\right.
\end{equation}
In fact, from  $(\ref{linjie})_1$, 
$(\ref{linjie})_3$ and $(\ref{linjie})_4$, we see
\begin{equation*}
\left\{
    \begin{array}{l}
       W''(\phi^*-0)=(R^*)^2g(\sigma_Q)-\displaystyle\frac{2}{\phi^*}R^*\Phi(\phi^*R^*),
       \\[0.3cm]
          W''(\phi^*+0)=(R^*)^2f(\sigma_Q)-\displaystyle\frac{2}{\phi^*}R^*\Phi(\phi^*R^*),  
          \\[0.3cm]
z'(\phi^*\pm0)=\bar\zeta\Phi(\phi^*R^*)+R^*\partial_{\bar\sigma}\Big(\Phi(\phi^*R^*)\Big)-\bar\xi W''(\phi^*\pm0).   
    \end{array}
    \right.
\end{equation*}
 Then  from $(\ref{linjie})_4$ and $(\ref{independence})$, we get  $(\ref{fanzheng})_3$ and $(\ref{fanzheng})_5$. Similarly, the others of $(2.40)$ can be checked.

 Based on a standard contradiction argument and maximum principle, we show that 
\begin{equation}\label{dandiao4}
  z(s)=\frac{\partial \mathcal{W}}{\partial \bar\sigma}(s,\bar\sigma)> 0\quad \mbox{for} \;\; 0<s<1.
\end{equation}
We first prove that $z(s)>0$ for $s\in(0,\phi^*)$.
By
$(\ref{fanzheng})_1$--$(\ref{fanzheng})_2$, we obtain 
$$
3z''(0)=2R^*(\bar\sigma)g(\sigma_D)\bar\zeta>0.
$$ 
If there exists a  $s_0\in(0,\phi^*)$ such that $z(s_0)\le0$, then $z(s)$ must attain a positive maximum in $(0,s_0)$, which contradicts $(\ref{fanzheng})_1$.  Therefore,
\begin{equation*}
    z(s)=\frac{\partial\mathcal{W}}{\partial\bar\sigma}(s,\bar\sigma)>0 \quad \mbox{for} \;\; 0<s<\phi^*.
\end{equation*}
By $(\ref{2.11})$ and assumption $(A3)$, we have 
\begin{equation}\label{FUhAO}
    \displaystyle\frac{2}{\phi^*}R^*\Phi(\phi^*R^*)-(R^*)^2f(\sigma_Q)\le\frac{2}{3}(R^*)^2g(\sigma_Q)-(R^*)^2f(\sigma_Q)<0.
\end{equation}
With $(\ref{2.11})$, $\bar\xi<0$, $\bar\zeta>0$ and $(\ref{FUhAO})$,  there hold
$$
z(\phi^*)>0,\qquad z'(\phi^*+0)>0,\qquad z(1)=1.
$$
If there exists some $s_1\in(\phi^*,1)$ such that $z(s_1)\le0$, then $z(s)$ attains a positive maximum in the interval $(\phi^*,s_1)$. Combining $(\ref{fanzheng})_4$, there is a contradiction, which implies that $z(s)>0$ for $s\in(\phi^*,1)$.
This completes the proof of the assertion $(i)$.

According to (\ref{xi zeta}), $(\ref{dandiao4})$, 
$(A2)$ and $(A3)$, we have
$$
\mathcal{G}'(\bar\sigma)=\int_{\phi^*}^1
S'(\mathcal{W})z(s)s^2ds-\Big(S(\sigma_Q)+\nu_1\Big)(\phi^*)^2\bar\xi>0  \qquad \mbox{for}\;\; \bar\sigma>\sigma_Q.
$$
 We get assertion $(ii)$. The proof is complete.
\end{proof}
\medskip

\begin{lem}\label{lem2.5} Let $\bar\sigma>\sigma_Q$.
  Under assumptions $(A1)$--$(A3)$, there
  exists a critical nutrient value $\sigma^{*} \in (\tilde{\sigma}, +\infty)$ such that equation $F(R, \bar{\sigma}) = 0$ has a unique root $R_{s}=R_s(\bar\sigma) \in (R^{*}(\bar{\sigma}), +\infty)$ 
if and only if  $\bar{\sigma} > \sigma^{*}$.  
\end{lem}
\begin{proof} 
 From Lemma \ref{lem2.3}, we see that for every given $\bar\sigma>\sigma_Q$,
 \begin{equation*}
     \sigma_Q<\mathcal{W}(s,\bar\sigma)<\bar\sigma \quad \mbox{for}\;\;\phi^*<s<1,
 \end{equation*}
  which with $(A2)$ and $(A3)$  implies that  $\mathcal G(\tilde\sigma)<g(\tilde \sigma)=0$.  Moreover, by $(\ref{2.37})$, the definition of $\mathcal{G}$ in Lemma \ref{2.4} $(ii)$ and variable transformation, we find 
  $$
\lim\limits_{\bar\sigma\to+\infty}\big(R^*(\bar\sigma)\big)^3\mathcal{G}(\bar\sigma)=\lim\limits_{R^*\to+\infty}\Big(\int_{\eta^*}^{R^*}S(U_2(r,\eta^*))r^2dr-\frac{\nu_1}{3}(\eta^*)^3\Big),
  $$
 which combined with $(A2)$, implies
   that $\mathcal G(+\infty)>0$.  Thus there exists a unique  $\sigma^*>\tilde\sigma$ such that 
\begin{equation}\label{2.43}
 F(R^*(\bar\sigma),\bar\sigma)=\mathcal G(\bar\sigma)
  \left\{
  \begin{array}{ll}
  >0,\qquad\mbox{for}\;\;  \bar\sigma>\sigma^*,
  \\ 
   =0, \qquad\mbox{for}\;\; \bar\sigma=\sigma^*,
   \\ 
   <0,\qquad\mbox{for}\;\; \sigma_Q<\bar\sigma<\sigma^*.
\end{array}
\right.
\end{equation}
By $(\ref{F'})$, $(\ref{2.30})$, $(\ref{2.43})$, we conclude that for $\bar\sigma>\sigma_Q$,  the equation $F(R,\bar\sigma)=0$   has a unique solution 
 $R_s\in( R^*(\bar\sigma),+\infty)$ if and only if $\bar\sigma>\sigma^*$. 
\end{proof}
\medskip

Next, we solve problem $(\ref{2.3})$  and  problem $(\ref{2.2})$.
The two-layer problem $(\ref{2.2})$ without
the condition $\sigma(0)\ge\sigma_D$ has 
been well studied in \cite{wu-xu}.  Recalling $(\ref{shoot})$--$(\ref{eta*})$, $(\ref{shoot2})$ and $(\ref{R*})$, for every fixed $\bar\sigma>\sigma_Q$,  we know that $R^*=R^*(\bar\sigma)$ is a threshold radius for the two-layer  tumor structure with $\sigma(0)=\sigma_D$. For $R>R^*(\bar\sigma)$, the tumor has a three-layer structure with a necrotic core. Similarly, by analyzing the problem $(\ref{2.2})_1$--$(\ref{2.2})_4$, we can prove that for every $\bar\sigma>\sigma_Q$,     there exists another critical tumor radius $R_*=R_*(\bar\sigma)$ corresponding to the one-layer tumor with only proliferating cells and $\sigma(0)=\sigma_Q$. 
\medskip

 \begin{lem}\label{lem2.6}
Under  assumption $(A1)$ and $\bar\sigma>\sigma_Q$,  there exists a unique critical radius $R_*=R_*(\bar\sigma)\in(0,R^*(\bar\sigma))$ which is strictly increasing in $\bar\sigma$ such that

$(i)$ For any $R\in(R_*(\bar\sigma),R^*(\bar\sigma)]$, there exists   a unique solution $(\sigma_1(r,R,\bar\sigma),\eta(R,\bar\sigma))$ of problem $(\ref{2.2})_1$--$(\ref{2.2})_4$.

$(ii)$ For any  $R\in(0,R_*(\bar\sigma)]$, there exists   a unique solution 
$\sigma_2(r,R,\bar\sigma)$  of problem $(\ref{2.3})_1$--$(\ref{2.3})_2$.
 \end{lem}

 \begin{proof}
     For any given $\eta\ge0$, we consider
\begin{equation}\label{shangxiajie}
\left\{
\begin{array}{l}
  \displaystyle u''(r)+\frac{2}{r}u'(r) = g(u(r))
  \qquad\mbox{for}\;\; 0<r<\eta,
\\[0.3 cm]
u'(0)=0,\qquad u(\eta)=\sigma_Q.
\end{array}
\right.
\end{equation}
Since $0$ and $\sigma_Q$ are the lower and upper solutions to  problem $(\ref{shangxiajie})$, respectively,  we have from the upper and lower solution method that problem $(\ref{shangxiajie})$ admits a unique solution $u=U_3(r,\eta)$ defined on $[0,\eta]$ with $0\le U_3(r,\eta)\le\sigma_Q$ for $r\in[0,\eta]$. By integrating $(\ref{shangxiajie})$, $u=U_3(r,\eta)$ is strictly increasing in $r$  for every $\eta\ge0$. Define 
$$
\widetilde\Phi(\eta):=(U_3)_r(\eta,\eta)=\frac{1}{\eta^2}\int_0^\eta g(U_3(\tau,\eta))\tau^2d\tau \qquad\mbox{for}\;\; \eta>0.
$$
 Following analogous arguments as $(\ref{2.11})$--$(\ref{2.15})$, we derive that $U_3(r,\eta)$ is strictly decreasing in $\eta$,   and (cf. Lemma 2.1 in \cite{wu-xu}),
\begin{equation}\label{2.45}
0<\widetilde\Phi(\eta)<\frac{1}{3}g(\sigma_Q)\eta\quad \mbox{and} \quad \widetilde\Phi'(\eta)>0\quad \text{for} \;\; \eta>0.    
\end{equation}
  By the monotonicity of $U_3(r,\eta)$ in $\eta$, $(\ref{shoot})$ and $(\ref{eta*})$,   $U_3(0,\eta)\ge\sigma_D$ if and only if $0<\eta\le\eta^*$. Particularly, $U_3(0,\eta^*)=\sigma_D$.

 We conclude that for every $\eta\in(0,\eta^*]$ problem $(\ref{2.2})_1$--$(\ref{2.2})_2$ possesses a unique solution $u=U_3(r,\eta)$ for $0\le r\le\eta$. 

By $(\ref{2.45})$ and continuity, we set $\widetilde\Phi(0)=0$.
 To solve problem $(\ref{2.2})_3$--$(\ref{2.2})_4$,  we next consider for  given $\eta\in[0,\eta^*]$  the following initial value problem:
    \begin{equation*}
    \left\{
        \begin{array}{l}
        u''(r)+\displaystyle\frac{2}{r}u'(r)=f(u(r))\qquad \mbox{for} \;\; r>\eta,
        \\[0.3cm]
    u(\eta)=\sigma_Q,
    \qquad u'(\eta)=\widetilde\Phi(\eta),
        \end{array}
        \right.
    \end{equation*}
    which admits a unique  solution $u(r)=U_4(r,\eta)$ on $[\eta,+\infty)$ strictly increasing in $r$ and decreasing in $\eta$ with 
    \begin{equation}\label{2.46}
        \lim\limits_{r\to +\infty}U_4(r,\eta)=+\infty \quad \mbox{for}\;\; 0\le\eta\le\eta^*,
    \end{equation}
ensured by a   proof  analogous to that in Lemma \ref{lem2.2}. 
  For any given $\bar\sigma>\sigma_Q$ and $0\le\eta\le\eta^*$, there thus exists a unique  $R(\eta,\bar\sigma)>0$ such that 
  \begin{equation}\label{shoot3}
      U_4(R(\eta,\bar\sigma),\eta)=\bar\sigma.
  \end{equation}
  From the monotonicity of $U_4(r,\eta)$ in $r$ and $\eta$ and the above equation, we can see that $R(\eta,\bar\sigma)$ is strictly increasing in $\eta$ for  $0\le\eta\le\eta^*$.
    Define  
   \begin{equation}\label{R_*}
R_*(\bar\sigma)=R(\eta,\bar\sigma)\Big|_{\eta=0},
   \end{equation}
   which is the critical radius such that problem $(\ref{2.3})_1$--$(\ref{2.3})_2$ has a unique solution $\sigma(r)$ satisfying $\sigma(0)=\sigma_Q$. Utilizing $(\ref{R*})$ and the monotonicity of $R(\eta,\bar\sigma)$ in $\eta$ for $0\le\eta\le\eta^*$, we have 
   \begin{equation}\label{R^*>R_*} R^*(\bar\sigma)=R(\eta^*,\bar\sigma)>R(\eta,\bar\sigma)\Big|_{\eta=0}=R_*(\bar\sigma).
   \end{equation}
   Similarly, by regarding $R$ and $\bar\sigma$ as  the variables and letting
   $$
  \eta=\eta(R,\bar\sigma),\quad \widetilde U_3(r,R,\bar\sigma)=U_3(r,\eta(R,\bar\sigma)),\quad \widetilde U_4(r,R,\bar\sigma)=U_4(r,\eta(R,\bar\sigma)),
   $$
   \begin{equation*}
       \sigma_1(r,R,\bar\sigma)=\left\{
       \begin{array}{l}
           \widetilde U_3(r,R,\bar\sigma) \quad \mbox{for}\;\; 0\le r\le\eta(R,\bar\sigma), 
           \\[0.3cm]
           \widetilde U_4(r,R,\bar\sigma) \quad  \mbox{for}\;\; \eta(R,\bar\sigma)< r\le R,
       \end{array}
       \right.
   \end{equation*}
   we see $(\sigma_1(r,R,\bar\sigma),\eta(R,\bar\sigma))$ solves problem $(2.2)_1$--$(\ref{2.2})_4$ for every
  $\bar\sigma>\sigma_Q$ and 
   $R\in(R_*(\bar\sigma),R^*(\bar\sigma)]$.  
   The proof of $R_*'(\bar\sigma)>0$  is similar as 
   that of (2.19) in \cite{wu-xu}. Hence we obtain  assertion $(i)$.
   
  The proof of assertion $(ii)$ is similar and simpler, we omit it here.
The proof is complete.
\end{proof}

 Now for any $\bar\sigma>\sigma_Q$, we extend the domain of $F(R,\bar\sigma)$ in (\ref{F(R)}) as follows: 
\begin{equation}\label{F}
    F(R,\bar\sigma):=\left\{
    \begin{array}{ll}
          \displaystyle\frac{1}{R^3}\Big[\int_{\eta(R,\bar\sigma)}^RS(\sigma_1(r,R,\bar\sigma))r^2dr-\frac{\nu_1}{3}\eta^3(R,\bar\sigma)\Big] \quad \mbox{for}\;\; R\in(R_*(\bar\sigma),R^*(\bar\sigma)], 
          \\[0.4cm]
         \displaystyle\frac{1}{R^3}\int_0^RS(\sigma_2(r,R,\bar\sigma))r^2dr  \qquad\mbox{for} \;\; R\in(0,R_*(\bar\sigma)].
    \end{array}
    \right.
\end{equation} 

\medskip

\begin{lem}\label{lem2.7}
  Under  assumptions $(A1)$--$(A3)$ and $\bar\sigma>\sigma_Q$, 
  the following assertions hold:

$(i)$  
 $\displaystyle
    \frac{\partial F(R,\bar\sigma)}{\partial R}<0$ for $R\in(0,R_*(\bar\sigma))\cup(R_*(\bar\sigma),
    R^*(\bar\sigma))$.
Moreover,
\begin{equation}\label{fuhao}
  \lim_{R \to 0^+}F(R,\bar\sigma)=\frac{1}{3}g(\bar\sigma)
  \left\{
  \begin{array}{l}
  >0,\qquad\mbox{for}\;\; \bar\sigma>\tilde\sigma,
  \\
  \le 0,\qquad \mbox{for}\;\; \sigma_Q<\bar\sigma\le\tilde\sigma.
  \end{array}
  \right.
\end{equation}

$(ii)$ There exists a critical nutrient value $\sigma_*\in(\tilde\sigma,\sigma^*)$ such that 
\begin{equation}\label{2.52}
 \mathcal{F}(\bar\sigma):=F(R_*(\bar\sigma),\bar\sigma)
  \left\{
  \begin{array}{ll}
  >0,\qquad \mbox{for}\;\; \bar\sigma>\sigma_*,
  \\ 
   =0, \qquad\mbox{for}\;\; \bar\sigma=\sigma_*,
   \\ 
   <0,\qquad \mbox{for}\;\; \sigma_Q<\bar\sigma<\sigma_*.
\end{array}
\right.
\end{equation}
Moreover, $\mathcal{F}(\bar\sigma)$ is strictly increasing on $(\sigma_Q,+\infty)$.
\end{lem}

\begin{proof}
 By a slight modification of the proofs of 
 (2.17), (2.20)-(2.23) in \cite{wu-xu}, we get the assertion
 $(i)$ and $(ii)$ with $\sigma_*>\tilde\sigma$. We only need to prove that $\sigma_*<\sigma^*$. 
    Combining $(\ref{F'})$, $(\ref{F})$ and assertion $(i)$ in this Lemma, we can derive that for every $\bar\sigma>\sigma_Q$, $F(R,\bar\sigma)$ is strictly
    decreasing in $R\in(0,+\infty)$. Then from $(\ref{2.43})$ $(\ref{R^*>R_*})$, $(\ref{2.52})$ and Lemma $\ref{lem2.4}$ $(ii)$, we have 
$$
\mathcal{G}(\sigma_*)=F(R^*(\sigma_*),\sigma_*)<F(R_*(\sigma_*),\sigma_*)=0=F(R^*(\sigma^*),\sigma^*)=\mathcal{G}(\sigma^*),
$$
which implies that $\tilde\sigma<\sigma_*<\sigma^*$. 
The proof is complete.
\end{proof}
\medskip

 With the above preparations, now we can state our main result on the existence and uniqueness of stationary solutions. 
 
  \medskip

\begin{thm}\label{thm2.9}
Under assumptions $(A1)$--$(A3)$, the existence and structure of the stationary solutions to problem $(\ref{1.1})$--$(\ref{1.4})$ are classified by the external nutrient supply $\bar\sigma$ as follows:

  $(i)$ For $\bar\sigma>\sigma^*$, there exists a unique stationary 
  solution  $(\sigma_s(r),\eta_s, \rho_s, R_s)$ with a three-layer structure to problem $(\ref{1.1})$--$(\ref{1.4})$,  where 
  $\eta_s=\eta(R_s,\bar\sigma)$, $\rho_s=\rho(R_s,\bar\sigma)$  and $R_s$  correspond to  the radius of the interface between proliferating cells and quiescent cells,  the radius of the necrotic core and 
  the unique root of $F(R,\bar\sigma)=0$ within $(R^*(\bar\sigma),+\infty)$, respectively.

  $(ii)$ For $\sigma_*<\bar\sigma\le \sigma^*$,  there exists  a unique stationary solution $(\sigma_s(r),\eta_s,R_s)$ with
  a two-layer structure  to problem
  $(\ref{1.1})$--$(\ref{1.4})$, where $\eta_s=\eta(R_s,\bar\sigma)$ and $R_s$ correspond to the radius of the quiescent core and the unique root of $F(R,\bar\sigma)=0$ within $(R_*(\bar\sigma),R^*(\bar\sigma)]$, respectively.

  $(iii)$ For $\tilde\sigma<\bar\sigma\le\sigma_*$,  there exists a unique stationary solution $(\sigma_s(r),R_s)$ with a one-layer structure containing only proliferating cells to problem $(\ref{1.1})$--$(\ref{1.4})$,
 where $R_s$ is the unique root of $F(R,\bar\sigma)=0$ within $(0,R_*(\bar\sigma)]$.

  $(iv)$ For $\bar\sigma\le \tilde\sigma$,  there exists only trivial stationary solution to problem $(\ref{1.1})$--$(\ref{1.4})$. 
\end{thm}

\begin{proof}

By the definition (\ref{F(R)}) and (\ref{F}) of $F(R,\bar\sigma)$, we see for $\bar\sigma>\sigma_Q$,
problem $(\ref{1.1})$--$(\ref{1.4})$ in 
the stationary case is equivalent to equation  
$F(R,\bar\sigma)=0$. 

From Lemma $\ref{lem2.5}$,
equation $F(R, \bar{\sigma}) = 0$ has a unique root $R_{s}=R_s(\bar\sigma) \in (R^{*}(\bar{\sigma}), +\infty)$ 
if and only if  $\bar{\sigma} > \sigma^{*}$. Then by the deduction before Lemma 2.3, we have $(\sigma_s, \rho_s, \eta_s, R_s)=(\Sigma(r, R_s, \bar\sigma), \rho(R_s,\bar\sigma),\eta(R_s, \bar\sigma), R_s)$ is the 
stationary solution of problem (1.1)--(1.4). Hence we get 
assertion $(i)$. 

Similarly,  by Lemma $\ref{lem2.7}$, we
see equation $F(R,\bar\sigma)=0$
has a unique root $R_s\in(R_*(\bar\sigma),R^*(\bar\sigma)]$ if and only if $\bar\sigma\in(\sigma_*,\sigma^*]$, and
equation $F(R,\bar\sigma)$=0
has a unique root $R_s\in(0,R_*(\bar\sigma)]$ if and only if
$\bar\sigma\in(\tilde\sigma,\sigma_*]$, then the assertion
$(ii)$ and $(iii)$ follows.

Finally, note that for $\bar\sigma<\tilde\sigma$, by $(A1)$ and $(A2)$, we can easily show $S(\sigma(r,t))<S(\tilde\sigma)=0$ for $0<r<R(t)$ and $t>0$. Then by (\ref{1.3}), we see that
$R'(t)<0$ for $R(t)>0$. Hence problem $(\ref{1.1})$--$(\ref{1.4})$ has  no non-trivial stationary solution.
The proof is complete.
\end{proof}

\medskip

\begin{rem}
     The above conditions on the existence and uniqueness
     of stationary solutions are actually necessary and sufficient, thereby the two critical nutrient concentrations $\sigma^*$ and $\sigma_*$ are biologically significant. From (\ref{shoot}), (\ref{eta*}), (\ref{R*}), (\ref{2.43}) and Lemma \ref{lem2.4}, we see that the critical 
nutrient concentration  $\sigma^{*}$  is determined by the functions
$f$, $g$, $S$ and the model parameters $\sigma_{Q}$, $\sigma_{D}$ and 
$\nu_{1}$. Similarly, from Lemmas \ref{lem2.6} and \ref{lem2.7}, we see $\sigma_{*}$ is determined by functions $f$, $S$ and the model parameter $\sigma_{Q}$.   From Lemma \ref{lem2.4} $(ii)$, 
we easily verify that
$\mathcal{G}$ is strictly decreasing in $\nu_1$ with any fixed $\bar\sigma>\tilde\sigma$ and thus $\sigma^*$ is  strictly increasing in $\nu_1$.  Similarly, we can consider the critical problem $(\ref{2.3})_1$--$(\ref{2.3})_2$ for $R=R_*(\bar\sigma)$ and use the analysis in Lemmas $\ref{lem2.1}$--$\ref{lem2.4}$ to derive that   $\sigma_*$ is strictly decreasing with respect to $\sigma_Q$. 

According to Theorem \ref{thm2.9} and  (\ref{2.43}), we see that $\sigma^*$ is the maximum  external nutrient concentration such that  the dormant tumor has  a two-layer structure with a quiescent core and outer a shell of proliferating cells, where the nutrient concentration at the tumor center is exactly $\sigma_D$. Similarly, by $(\ref{2.52})$, $\sigma_*$ is the maximum  external nutrient concentration such that  the dormant tumor has  a one-layer structure with only proliferating cells, where the nutrient concentration at the tumor center is exactly $\sigma_Q$. The two values $\sigma^{*}$ and $\sigma_{*}$ correspond to the critical states between the three-layer and two-layer, and between the two-layer and one-layer dormant tumor structures. They will play a significant role in the long-time dynamics of solid tumors.
\end{rem}

 \medskip
 \hskip 1em

\section{Asymptotic behavior}\label{3}
 \setcounter{equation}{0}
\hskip 1em

  In this section, we study the asymptotic behavior of transient solutions of 
  free boundary problem $(\ref{1.1})$--$(\ref{1.4})$.

\medskip

\begin{thm}\label{thm3.1}
    Under assumptions $(A1)$--$(A3)$,  for any $\bar\sigma>0$ and $R_0>0$, problem $(\ref{1.1})$--$(\ref{1.4})$ has a unique global solution for all
    $t\ge 0$, and the following asymptotic behavior holds:
 
  $(i)$ If $\bar\sigma\le \tilde\sigma$, 
  then for any $R_0>0$, $\displaystyle \lim_{t\to+\infty}R(t)=0$, 
  and the tumor will finally disappear. 
    
  $(ii)$ If $\tilde\sigma<\bar\sigma\le\sigma_*$, then for any $R_0>0$, 
   $\displaystyle\lim_{t\to+\infty}R(t)=R_s$, and the tumor will 
   finally converge to the proliferating one-layer stationary state with
    radius $R_s$.
   
   $(iii)$ If  $\sigma_*<\bar\sigma\le\sigma^*$, then for any $R_0>0$, 
   $$\displaystyle\lim_{t\to+\infty}R(t)=R_s, \quad \lim_{t\to+\infty}\eta(t):=\lim_{t\to+\infty}\eta(R(t),\bar\sigma) =\eta_s,
   $$
   and the tumor will finally converge to the 
   proliferating-quiescent two-layer stationary state with radius $R_s$, 
   and a quiescent core whose  radius is $\eta_s$. 

$(iv)$ If  $\bar\sigma>\sigma^*$, then for any $R_0>0$, 
   $$ \displaystyle\lim_{t\to+\infty}R(t)=R_s,\quad \lim_{t\to+\infty}\eta(t):=\lim_{t\to+\infty}\eta(R(t),\bar\sigma) =\eta_s,\quad \displaystyle\lim_{t\to+\infty}\rho(t):=\lim_{t\to+\infty}\rho(R(t),\bar\sigma) =\rho_s,
   $$
    and the tumor will finally converge to the 
   proliferating-quiescent-necrotic three-layer stationary state with radius $R_s$, 
   a necrotic core with radius $\rho_s$ and an interface $r=\eta_s$ which separating the proliferating cells and quiescent cells.   
\end{thm}

 \begin{proof}
     By  the proof of Theorem 2.8 $(iv)$, we immediately get the assertion $(i)$.
     
     From Lemmas \ref{lem2.1}--\ref{lem2.2} and \ref{lem2.6}, we conclude for any given $\bar\sigma>\sigma_Q$, $R(t)>0$ for $t>0$, there exists a unique nutrient concentration function $\sigma=\sigma(r,t)$  solving problem $(\ref{1.1})$--$(\ref{1.2})$. Then problem $(\ref{1.1})$--$(\ref{1.4})$ is equivalent to the following Cauchy problem:
     \begin{equation}\label{3.1}
     \left\{
         \begin{array}{ll}
              \displaystyle\frac{d R(t)}{dt}=R(t) F(R(t),\bar\sigma) \qquad \mbox{ for}\;\; t>0, 
              \\[0.3cm]
              R(0)=R_0,
         \end{array}
         \right.
     \end{equation}
     where $F(R,\bar\sigma)$ is given by $(\ref{F(R)})$ and $(\ref{F})$.
     
     Clearly, $F(R,\bar\sigma)$ is strictly decreasing in $R>0$ for any fixed $\bar\sigma>\sigma_Q$,
    and by $(A3)$ we have the following uniform estimate
      \begin{equation}\label{3.2}
         -\frac{\nu_2}{3}\le F(R,\bar\sigma)\le 
         \frac{S(\bar\sigma)}{3}\qquad
         \mbox{for}\;\; R>0.
     \end{equation}
     It implies that 
     \begin{equation}\label{3.3}
        R_0 e^{-\frac{\nu_2}{3}t} \le R(t)\le 
        R_0e^{\frac{S(\bar\sigma)}{3}t} \qquad \mbox{for} \;\; t>0.
     \end{equation}
Consequently, problem $(\ref{3.1})$ has a unique solution for all $t>0$.

By Theorem \ref{thm2.9} $(i)$--$(iii)$ and 
the classical theory of 
ordinary differential equations, we see that for any 
$\bar\sigma>\tilde\sigma$, the unique stationary
solution $R_s$ of problem (\ref{3.1}) is globally asymptotically stable. Moreover, for any initial
data $R_0>0$, $R(t)$ converges exponentially to 
$R_s$. Then by the deduction in
Section 2, we get the desired results in assertions
$(ii)$--$(iv)$.
 The proof is complete.
\end{proof}

\medskip

From Lemma \ref{lem2.3} and Lemma \ref{lem2.6}, we see that 
during the tumor evolution, the internal structure of the tumor is determined  by the external nutrient concentration $\bar\sigma$ and the tumor radius $R(t)$. The tumor may exhibit the following six distinct states during its evolution:

$(1)$ proliferating one-layer state; 

$(2)$ proliferating-quiescent two-layer state; 

$(3)$ proliferating-quiescent-necrotic three-layer state;

$(4)$ quiescent one-layer state;

$(5)$ quiescent-necrotic two-layer state;    

$(6)$ necrotic one-layer state.

More precisely,  Lemma \ref{lem2.3} and Lemma $\ref{lem2.6}$ imply that if $\bar\sigma>\sigma_Q$, there exist two positive 
critical radii $R_*(\bar\sigma)$ and $R^*(\bar\sigma)$
($R_*(\bar\sigma)<R^*(\bar\sigma)$), such that the tumor stays in the proliferating one-layer state for
 $0<R\le R_*(\bar\sigma)$,  in the proliferating-quiescent two-layer state for 
 $R_*(\bar\sigma)<R\le R^*(\bar\sigma)$, and in proliferating-quiescent-necrotic three-layer state for $R>R^*(\bar\sigma)$. 

 On the other hand, if $\sigma_D<\bar\sigma\le \sigma_Q$, the tumor will not contain any proliferating cells, and
 it may exhibit quiescent one-layer state or 
 quiescent-necrotic two-layer state. If 
 $\bar\sigma\le\sigma_D$, then all tumor cells
 are obviously necrotic and the tumor always stays
 in the necrotic one-layer state.

Combining with Theorem 3.1,  
mutual transformation between these different structural states can be observed during the tumor evolution. Especially, we can see the formation and dissolution of the quiescent and necrotic cores.
\medskip

\begin{cor}\label{cor3.2} Assume $(A1)$--$(A3)$ hold and $\bar\sigma>\sigma_Q$. We have the following assertions: 

$(i)$ For $\bar\sigma>\sigma^*$ and initial radius $R_*(\bar\sigma)<R_0< R^*(\bar\sigma)$, 
 there exists a  time point $T>0$   such that the tumor is in   proliferating-quiescent two-layer state for $0<t\le T$, and
  in proliferating-quiescent-necrotic three-layer state  for  $ t> T$.

  $(ii)$ 
  For $\sigma_*<\bar\sigma\le\sigma^*$ and  initial radius $R_0>R^*(\bar\sigma)$, there exists a  time point $T>0$  
  such that the tumor is in proliferating-quiescent-necrotic three-layer state for 
  $0<t<T$, 
  and in proliferating-quiescent two-layer state for $t\ge T$.

  $(iii)$ For $\sigma_*<\bar\sigma\le\sigma^*$ and initial radius  $0<R_0< R_*(\bar\sigma)$, there exists a time point $T>0$  such that the tumor is in proliferating one-layer state for $0<t\le T$, and in proliferating-quiescent two-layer state for $t> T$.

  $(iv)$ For $\sigma_Q<\bar\sigma\le\sigma_*$ and initial radius  $R_*(\bar\sigma)< R_0\le R^*(\bar\sigma)$, there exists  a  time point $T>0$  such that the tumor is in proliferating-quiescent two-layer state for $0<t<T$, and in proliferating one-layer state for $t\ge T$.

  $(v)$ For $\bar\sigma>\sigma^*$ and initial radius  $0<R_0< R_*(\bar\sigma)$, there exist two  time
  points $T_2>T_1>0$  such that the tumor is in proliferating one-layer state for $0<t\le T_1$, in proliferating-quiescent two-layer state for $T_1< t\le T_2$, and in proliferating-quiescent-necrotic three-layer state for $t> T_2$.

$(vi)$ For $\sigma_Q<\bar\sigma\le\sigma_*$ and initial radius  $R_0>R^*(\bar\sigma)$, there exist two  time points $T_2>T_1>0$  such that the tumor is in proliferating-quiescent-necrotic three-layer state for $0<t<T_1$,  in proliferating-quiescent two-layer state for $T_1\le t< T_2$, and in proliferating one-layer state for $t\ge T_2$.

\end{cor}

\begin{proof}
    We  only prove assertion $(vi)$, as the remaining assertions follow analogously. Suppose that $\sigma_Q<\bar\sigma\le\sigma_*$ and $R_0>R^*(\bar\sigma)$. By Lemma $\ref{lem2.3}$, we see that at the initial time $t=0$, the tumor has a proliferating-quiescent-necrotic three-layer structure. The inner two free boundaries are characterized by $\rho(t)\vert_{t=0}=\rho(R_0,\bar\sigma)$ and $\eta(t)\vert_{t=0}=\eta(R_0,\bar\sigma)$. Moreover,
    $F(R_0,\bar\sigma)<0$. By the monotonicity of $F(R,\bar\sigma)$ in $R$,  the tumor radius $R(t)$ converges monotonically decreasing to the stationary radius $R_s\in[0,R_*(\bar\sigma)]$ as $t\to\infty$. Consequently, there exist two times $0<T_1<T_2$ such that $R(T_1)=R^*(\bar\sigma)$ and $R(T_2)=R_*(\bar\sigma)$, which means that when $0<t<T_1$, the tumor is in proliferating-quiescent-necrotic three-layer state, while the tumor is in proliferating-quiescent two-layer state in the time interval $T_1\le t<T_2$. When $t\ge T_2$, the tumor is in the proliferating one-layer state. This completes the proof of assertion $(vi)$.
\end{proof}

\begin{rem} In the case $\sigma_D<\bar\sigma\le\sigma_Q$, we can also observe
mutual transformation between the quiescent one-layer state and the quiescent-necrotic two-layer state. In fact, replacing the proliferating layer by the quiescent layer in the necrotic tumor model of
 \cite{wu-wang-19}, there exists a critical radius
 $R^*(\bar\sigma)$ for $\sigma_D<\bar\sigma\le \sigma_Q$, such that 
 the tumor is in  the quiescent-necrotic two-layer state for $R>R^*(\bar\sigma)$, while in the quiescent one-layer state for $0<R\le R^*(\bar\sigma)$.   
    Then for any initial radius $R_0>R^*(\bar\sigma)$, there exists a finite time point $T>0$ such that the tumor is in the quiescent-necrotic two-layer state for $0<t<T$,  
    while in the quiescent one-layer state 
    for $t\ge T$ and will finally disappear.
\end{rem}
 \medskip
 \hskip 1em

\section{Conclusions and biological discussion}\label{4}
 \setcounter{equation}{0}
\hskip 1em
\textit{}

In this paper, we investigate  a free boundary problem modeling  the growth of solid tumors with three layers, a configuration with significant biological relevance observed in experiments
\cite{byr-cha-97,2021,Han-2021}. 
We establish a complete classification of radial
stationary solutions and long time behavior of radial transient solutions. By using a nonlinear analysis  approach, we find 
two critical nutrient concentration values $\sigma^*$ and $\sigma_*$, such that if the external nutrient supply $\bar\sigma>\sigma^*$, the tumor will
exponentially evolve to the three-layer stationary solution and if 
$\sigma_* <\bar\sigma<\sigma^*$ 
or $\tilde\sigma<\bar\sigma<\sigma_*$, the tumor
eventually evolves to the 
two-layer stationary solution or the one-layer stationary solution, respectively.   Furthermore, comparing with the two-layer tumor models (cf. \cite{liu-zhuang, wu-wang-19, wu-xu}), we obtain a new critical radius $R^*(\bar\sigma)$ for $\bar\sigma>\sigma_Q$,
 which distinguishes between the three-layer tumor structure and the two-layer tumor structure. This is consistent with the experimental observation on tumor spheroids in \cite{Han-2021}: approximately beyond a critical size of  $500$ $\rm{\mu}$m, a three-layer structure forms inside the tumor spheroids.  Our analysis also reveals that the critical tumor radius $R^*(\bar\sigma)$ is strictly increasing in the external nutrient concentration $\bar\sigma$.

Our work demonstrates that the external nutrient concentration $\bar\sigma$ may be taken as a measurable parameter determining the tumor's development process and final structure. This insight offers a potential mechanism for controlling tumor structure by modulating external nutrient supply. We hope these results may be useful for future tumor studies. 

\bigskip

\vspace{0.25em}

\begin{ac*}
This research is supported by the National Natural Science Foundation of China under the Grant No. 12101260, 12171349 and 12271389.
\end{ac*}

\end{CJK}

{\small
\begin{thebibliography}{99}

\bibitem{bue-erc-08}  H. Bueno, G. Ercole and A. Zumpano,
  Stationary solutions of a model for the growth of tumors and 
  a connection between the nonnecrotic and necrotic phases,
  {\em SIAM J. Appl. Math.},
  {\bf 68} (2008), 1004--1025.
  

 \bibitem{byr-cha-96}  H. Byrne and M. Chaplain, Growth of necrotic tumors
 in the presence and absence of inhibitors, {\em Math. Biosci.}, {\bf 135}
 (1996), 187--216.
 
\bibitem{byr-cha-97}  H. Byrne and M. Chaplain, 
   Free boundary problems associated with the growth and development of multicellular spheroids,
  {\em European J. Appl. Math.}, {\bf 8} (1997), 639--658.

  
\bibitem{cui-05} S. Cui, Analysis of a free boundary problem
  modeling tumor growth,  {\em Acta Math. Sinica}, English Series, {\bf 21} (2005),
   1071--1082.
  

\bibitem{cui-06}  S. Cui, Formation of necrotic cores in the 
  growth of tumors: analytic results, {\em Acta Math. Scientia}, {\bf 26} (2006), 781--796.
  

\bibitem{cui-esc-07}  S. Cui and J. Escher,  Bifurcation 
  analysis of an elliptic free boundary problem modeling the growth of avascular tumors, {\em SIAM J. Math. Anal.}, {\bf 39} (2007), 210--235.

\bibitem{cui-esc-08}  S. Cui and J. Escher,  Asymptotic behavior of solutions of a multidimensional
moving boundary problem modeling tumor growth, {\em Comm. Partial Differential Equations}, {\bf 33} (2008), 636--655.

  \bibitem{cui-09} S. Cui, Lie group action and stability analysis of stationary   
  solutions for a free boundary problem modelling tumor growth, {\em J. Differential
   Equations},  {\bf 246} (2009), 1845--1882.


\bibitem{cui-fri-01}  S. Cui and A. Friedman, Analysis of a mathematical
  model of the growth of necrotic tumors, {\em J. Math. Anal. Appl.}, {\bf 255}
  (2001), 636--677.


  
   

\bibitem{fri-rei-99}   A. Friedman and F. Reitich, Analysis of a mathematical
 model  for the growth of tumors, {\em J. Math. Biol.}, {\bf 38} (1999), 262--284. 

  

\bibitem{fri-hu-06}  A. Friedman and B. Hu, Asymptotic stability for a free
  boundary problem arising in a tumor model, {\em J. Differential Equations},
  {\bf 227} (2006), 598--639.


\bibitem{fri-hu-08} A. Friedman and B. Hu, Stability and instability of
  Liapunov-Schmidt and Hopf bifurcation for a free boundary problem arising
  in a tumor model, {\em Trans. Amer. Math. Soc.}, {\bf 360} (2008), 5291--5342.
  
\bibitem{greenspan} H. P. Greenspan, Models for the growth of a solid tumor by diffusion,
{\em Stud. Appl. Math.}, {\bf 51} (1972),
317--340.

\bibitem{2021}
S. Gunti, A. T. Hoke, K. P. Vu and N. R. London Jr,
Organoid and spheroid tumor models: techniques and applications, 
{\em Cancers }, {\bf 13}(2021), 874.



\bibitem{Han-2021}
S. J. Han, S. Kwon and K. S. Kim, Challenges of applying multicellular tumor spheroids in preclinical phase, {\em Cancer Cell Int.}, {\bf 21} (2021), 152.

 
\bibitem{he-xing-hu}  W. He, R. Xing and B. Hu, Linear stability  analysis
  for a free boundary problem  modeling multilayer tumor growth with time delay,
  {\em SIAM J. Math. Anal.}, {\bf 54} (2022), 4238--4276.  


  
\bibitem{hua-hu-24} Y. Huang and B. Hu, Periodic solution for a
  free-boundary tumor model with small diffusion-to-growth ratio,
  {\em J. Differential Equations}, {\bf 399} (2024), 252--280.

\bibitem{hua-zha-hu-19}  Y. Huang, Z. Zhang and B. Hu, 
  Linear stability for a free boundary tumor model with a periodic supply 
  of external nutrients, {\em Math. Methods Appl. Sci.},  {\bf 42} (2019), 1039--1054.



\bibitem{liu-zhuang}  Y. Liu and Y. Zhuang, Analytic results of a 
    double-layered radial tumor model with different consumption rates,
    {\em Nonlinear Anal. Real World Appl.}, {\bf 76} (2024),  104004.

\bibitem{liu-z-2} Y. Liu and Y. Zhuang, Time delays in a double-layered radial tumor model with different living cells, {\em Math. Methods Appl. Sci.}, {\bf 48} (2025),  2655--2664.

\bibitem{liu-z-3} Y. Liu and Y. Zhuang, Analytic results on the evolution of a triple-layered radial tumor with jump discontinuous consumption rate, {\em J. Evol. Equations}, accepted for publication.



\bibitem{low} J. S. Lowengrub,  H. B. Frieboes, F. Jin, et al., Nonlinear modelling of cancer: bridging the gap between cells and tumours, {\em Nonlinearity}, {\bf 23} (2010), R1--91.



\bibitem{lu-hao-hu}  M. Lu, W. Hao, B. Hu and S. Li, Bifurcation analysis
of a free boundary model of vascular tumor growth with a necrotic core and chemotaxis, {\em J. Math. Biol.}, {\bf 86} (2023), 19.


\bibitem{perthame} B. Perthame, M. Tang and N. Vauchelet, Traveling wave solution of the 
Hele-Shaw model of tumor growth with nutrient,
{\em Math. Models Meth. Appl. Sci.}, {\bf 24} (2014), 2601--2626.



\bibitem{wu-18} J. Wu, Asymptotic stability of a free boundary problem for
  the growth of multi-layer tumours in the necrotic phase, {\em Nonlinearity},
  {\bf 32} (2019), 2955--2974.


\bibitem{wu-19} J. Wu,  Bifurcation for a free boundary problem modeling
  the growth of necrotic multilayered tumors, {\em Discrete Contin. Dyn. Syst.}, 
  {\bf 39} (2019),  3399--3411. 
  
\bibitem{wu-21} J. Wu, Analysis of a nonlinear necrotic tumor model with
  two free boundaries, {\em J. Dynam. Differential Equations}, {\bf 33} (2021),
  511--524.

\bibitem{wu-wang-19} J. Wu and C. Wang, Radially symmetric growth of necrotic tumors and connection with nonnecrotic tumors, {\em Nonlinear Anal. Real World Appl.},   {\bf 50} (2019), 25--33. 

\bibitem{wu-xu-20} J. Wu and S. Xu, Asymptotic behavior of a nonlinear 
  necrotic tumor model with a periodic external nutrient supply, 
  {\em Discrete Contin. Dyn. Syst. Ser. B}, 
  {\bf 25 (7)} (2020),  2453--2460. 
  
  \bibitem{wu-xu} J. Wu, H. Xu and Y. Zhuang,  Analysis of a nonlinear free boundary problem modeling the radial growth of two-layer tumors.
  {\em J. Nonlinear Sci., }\textbf(2025), 98.
  
\bibitem{xu-zhang-zhou} S. Xu, F. Zhang and Q. Zhou, A free boundary 
  problem for necrotic tumor growth with angiogenesis, {\em Appl. Anal.}, 
  {\bf 102} (2023), 977--987.

\bibitem{zhao2025} 
X. E. Zhao and J. Shi, 
On determination of the bifurcation type for a free boundary problem modeling tumor growth, 
 \emph{J. Differential Equations}, \textbf{436} (2025), 113352.
  
\bibitem{zheng-li-zhuang} J. Zheng, R. Li and Y. Zhuang, Analysis of the growth
  of a radial tumor with triple-layered structure, {\em Discrete Contin. Dyn. Syst.}, 
  {\bf 44} (2024), 1958--1981.



\end{thebibliography}}
\end{document}